\documentclass[a4, 10pt]{amsart}
\usepackage{amssymb}
\usepackage{amstext}
\usepackage{amsmath}
\usepackage{amscd}
\usepackage{amsthm}
\usepackage{amsfonts}
\usepackage{enumerate}
\usepackage{graphicx}
\usepackage{latexsym}

\theoremstyle{plain}
\newtheorem{thm}{Theorem}[section]
\newtheorem{lem}[thm]{Lemma}
\newtheorem{cor}[thm]{Corollary}

\newtheorem*{thm*}{Theorem}
\newtheorem*{cor*}{Corollary}
\newtheorem*{prop*}{Proposition}

\newtheorem*{claim*}{Claim}

\theoremstyle{definition}
\newtheorem{defn}[thm]{Definition}

\newtheorem*{conj*}{Conjecture}

\theoremstyle{remark}
\newtheorem*{mpf}{{\it Proof of Theorem \ref{main}}}

\numberwithin{equation}{thm}
\def\Hom{\operatorname{Hom}}

\def\add{\operatorname{add}}

\def\Ext{\operatorname{Ext}}

\def\Tor{\operatorname{Tor}}

\def\Ker{\operatorname{Ker}}
\def\Im{\operatorname{Im}}
\def\tr{\operatorname{Tr}}

\def\m{\mathfrak m}

\def\p{\mathfrak p}

\def\cdim{\operatorname{\text{$C$}-dim}}

\def\depth{\operatorname{depth}}

\def\grade{\operatorname{grade}}

\begin{document}

\title[Embeddings into modules of finite homological dimensions]{On the existence of embeddings into modules\\
of finite homological dimensions}
\author{Ryo Takahashi}
\address{Department of Mathematical Sciences, Faculty of Science, Shinshu University, 3-1-1 Asahi, Matsumoto, Nagano 390-8621, Japan}
\email{takahasi@math.shinshu-u.ac.jp}
\author{Siamak Yassemi}
\address{Department of Mathematics, University of Tehran, P. O. Box 13145-448, Tehran, Iran -- and -- School of Mathematics, Institute for
Research in Fundamental Sciences (IPM), P. O. Box 19395-5746, Tehran, Iran}
\email{yassemi@ipm.ir}
\author{Yuji Yoshino}
\address{Graduate School of Natural Science and Technology, Okayama University, Okayama 700-8530, Japan}
\email{yoshino@math.okayama-u.ac.jp}
\thanks{The first and second authors were supported in part by Grant-in-Aid for Young Scientists (B) 19740008 from JSPS and by grant No. 88013211 from IPM, respectively}
\begin{abstract}
Let $R$ be a commutative Noetherian local ring. We show that $R$ is Gorenstein if and only if every finitely generated $R$-module can be embedded in a finitely generated $R$-module of finite projective dimension. This extends a result of Auslander and Bridger to rings of higher Krull dimension, and it also improves a result due to Foxby where the ring is assumed to be Cohen-Macaulay.
\end{abstract}
\maketitle

\section{Introduction}

Throughout this paper, let $R$ be a commutative Noetherian local ring.
All $R$-modules in this paper are assumed to be finitely generated.

In \cite[Proposition 2.6 (a) and (d)]{ABr} Auslander and Bridger proved the following.

\begin{thm}[Auslander-Bridger]\label{ausbr}
The following are equivalent:
\begin{enumerate}[\rm (1)]
\item
$R$ is quasi-Frobenius (i.e. Gorenstein with Krull dimension zero).
\item
Every $R$-module can be embedded in a free $R$-module.
\end{enumerate}
\end{thm}

On the other hand, in \cite[Theorem 2]{F} Foxby showed the following.

\begin{thm}[Foxby]\label{foxby}
The following are equivalent:
\begin{enumerate}[\rm (1)]
\item
$R$ is Gorenstein.
\item
$R$ is Cohen-Macaulay, and every $R$-module can be embedded in an $R$-module of finite projective dimension.
\end{enumerate}
\end{thm}

For an $R$-module $C$ we denote by $\add_RC$ the class of
$R$-modules which are direct summands of finite direct sums of
copies of $C$. The {\em $C$-dimension} of an $R$-module $X$, $\cdim_RX$,
is defined as the infimum of nonnegative integers $n$ such that
there exists an exact sequence
$$
0 \to C_n \to C_{n-1} \to \cdots \to C_0 \to X \to 0
$$
of $R$-modules with $C_i\in\add_RC$ for all $0\le i\le n$.

In this paper, we prove the following theorem. This result removes from Theorem \ref{foxby} the assumption that $R$ is Cohen-Macaulay, and it extends Theorem \ref{ausbr} to rings of higher Krull dimension. It should be noted
that our proof of this result is different from Foxby's proof for
the special case $C=R$.

\begin{thm}\label{main}
Let $R$ be a commutative Noetherian local ring with residue field $k$.
Let $C$ be a semidualizing $R$-module of depth $t$.
Then the following are equivalent:
\begin{enumerate}[\rm (1)]
\item
$C$ is dualizing.
\item
Every $R$-module can be embedded in an $R$-module of finite $C$-dimension.
\item
The $R$-module $\tr\Omega^tk\otimes_RC$ can be embedded in an $R$-module of finite $C$-dimension.
(Here $\tr\Omega^tk$ denotes the transpose of the $t$-th syzygy of the $R$-module $k$.)
\end{enumerate}
Moreover, if one of these three conditions holds, then $R$ is Cohen-Macaulay.
\end{thm}

\section{Proof of Theorem \ref{main} and its applications}

First of all, we recall the definition of a semidualizing module.

\begin{defn}
An $R$-module $C$ is called {\em semidualizing} if the natural homomorphism $R\to\Hom_R(C,C)$ is an isormophism and $\Ext_R^i(C,C)=0$ for all $i>0$.
\end{defn}

Note that a dualizing module is nothing but a semidualizing module of finite injective dimension.
Another typical example of a semidualizing module is a free module of rank one.
Recently a considerable number of authors have studied semidualizing modules and have obtained many results concerning these modules.

We denote by $\m$ the maximal ideal of $R$ and by $k$ the residue field of $R$.
To prove our main theorem, we establish two lemmas.

\begin{lem}\label{syz}
Let $C$ be a semidualizing $R$-module.
Let $g: M\to X$ be an injective homomorphism of $R$-modules with $\cdim_RX<\infty$.
If $\Ext_R^i(M,C)=0$ for any $1\le i\le \cdim_RX$, then the natural map $\lambda_M:M\to\Hom_R(\Hom_R(M,C),C)$ is injective.
\end{lem}

\begin{proof}
First of all we prove that $M$ can be embedded in a module $C_0$  in $\add_RC$.
For this we set $n=\cdim_RX$. 
If $n =0$, then this is obvious from the assumption, since  $X \in \add_R C$.    If $n > 0$, then there exists an exact sequence
$$
0 \to C_n \overset{d_n}{\to} C_{n-1} \overset{d_{n-1}}{\to} \cdots \overset{d_1}{\to} C_0 \overset{d_0}{\to} X \to 0
$$
with $C_i\in\add_RC$ for $0\le i\le n$.
Putting $X_i=\Im d_i$, we have exact sequences 
$$
0 \to X_{i+1} \to C_i \to X_i \to 0  \quad  (0 \le i \le n-1). 
$$
Then we have $\Ext_R^1(M,X_1)=0$, since there are isomorphisms $\Ext_R^1(M,X_1)\cong\Ext_R^2(M,X_2)\cong\cdots\cong\Ext_R^n(M,X_n)\cong\Ext_R^n(M,C_n)=0$.
Hence  $\Hom _R(M, d_0) : \Hom _R (M, C_0) \to \Hom _R (M, X)$  is surjective. 
This implies that the homomorphism  $g \in \Hom _R (M, X)$  is lifted to $f \in \Hom _R (M, C_0)$, i.e.  $d_0 \cdot f = g$. 
Since  $g$  is injective, $f$ is injective as well. 
Therefore $M$ has an embedding $f$  into $C_0$.  

To prove that  $\lambda _{M}$  is injective, we note that 
$\lambda _{C_0}$  is an isomorphism, because of $C_0 \in \add _RC$. 
Since there is an injective homomorphism $f:M\to C_0$, the following 
 commutative diagram forces  $\lambda _M$  to be injective:
$$
\begin{CD}
M @>{f}>> C_0 \\
@V{\lambda_M}VV @V{\lambda_{C_0}}V{\cong}V \\
\Hom_R(\Hom_R(M,C),C) @>{\Hom_R(\Hom_R(f,C),C)}>> \Hom_R(\Hom_R(C_0,C),C).
\end{CD}
$$
\end{proof}

\begin{lem}\label{sen}
Let $C$ be a semidualizing $R$-module and let $M$ be an $R$-module.
Assume that $M$ is free on the punctured spectrum of $R$.
Then there is an isomorphism
$$
\Ext_R^i(M,R)\cong\Ext_R^i(M\otimes_RC,C)
$$
for each integer $i\le\depth_RC$.
\end{lem}

\begin{proof}
Set $t=\depth_RC$.
Since $C$ is semidualizing, we have a spectral sequence
$$
E_2^{p,q}=\Ext_R^p(\Tor_q^R(M,C),C)\Rightarrow\Ext_R^{p+q}(M,R).
$$
Note by assumption that the $R$-module $\Tor_q^R(M,C)$ has finite length for $q>0$.
By \cite[Proposition 1.2.10(e)]{BH} we have $E_2^{p,q}=0$ if $p<t$ and $q>0$.
Hence
$$
\Ext_R^i(M\otimes_RC,C)=E_2^{i,0}\cong\Ext_R^i(M,R)
$$
for $i\le t$.
\end{proof}

Let $M$ be an $R$-module.
Take a free resolution
$$
F_\bullet=(\cdots \overset{d_{n+1}}{\to} F_n \overset{d_n}{\to} F_{n-1} \overset{d_{n-1}}{\to} \cdots \overset{d_1}{\to} F_0 \to 0)
$$
of $M$.
Then for a nonnegative integer $n$ we define the {\em $n$-th syzygy} of $M$ by the image of $d_n$ and denote it by $\Omega_R^nM$ or simply $\Omega^nM$.
We also define the {\em (Auslander) transpose} of $M$ by the cokernel of the map $\Hom_R(d_1,R):\Hom_R(F_0,R)\to\Hom_R(F_1,R)$ and denote it by $\tr_RM$ or simply $\tr M$.
Note that the $n$-th syzygy and the transpose of $M$ are uniquely determined up to free summand.
Note also that they commute with localization; namely, for every prime ideal $\p$ of $R$ there are isomorphisms $(\Omega_R^nM)_\p\cong\Omega_{R_\p}^nM_\p$ and $(\tr_RM)_\p\cong\tr_{R_\p}M_\p$ up to free summand.

Recall that for a positive integer $n$ an $R$-module is called {\em $n$-torsionfree} if
$$
\Ext_R^i(\tr M,R)=0
$$
for all $1\le i\le n$.
Now we can prove our main theorem.

\begin{mpf}
(1) $\Rightarrow$ (2): By virtue of \cite[Theorem (3.11)]{S}, the local ring $R$ is Cohen-Macaulay.
Now assertion (2) follows from \cite[Theorem 1]{F}.

(2) $\Rightarrow$ (3): This implication is obvious.

(3) $\Rightarrow$ (1): We denote by $(-)^\dag$ the $C$-dual functor $\Hom_R(-,C)$.
Put $t=\depth_RC$ and set $M=\tr\Omega^tk$.
Then we have $\depth R=t$ by \cite{G}.
Since
$$
\grade_R\Ext_R^i(k,R)\ge i-1
$$
for $1\le i\le t$, the module $\Omega^tk$ is $t$-torsionfree by \cite[Proposition (2.26)]{ABr}.
Hence $\Ext_R^i(M,R)=0$ for $1\le i\le t$.
As $M$ is free on the punctured spectrum of $R$, Lemma \ref{sen} implies $\Ext_R^i(M\otimes_RC,C)=0$ for $1\le i\le t$.
By assumption (3), the module $M\otimes_RC$ has an embedding into a module $X$ with $\cdim_RX<\infty$.
According to \cite[Lemma 4.3]{tor}, we have $\cdim_RX\le t$.
Lemma \ref{syz} shows that the natural map $\lambda_{M\otimes_RC}:M\otimes_RC\to (M\otimes_RC)^{\dag\dag}$ is injective.
On the other hand, since there are natural isomorphisms
\begin{align*}
(M\otimes_RC)^{\dag\dag} & =\Hom_R(\Hom_R(M\otimes_RC,C),C)\cong\Hom_R(\Hom_R(M,\Hom_R(C,C)),C)\\
& \cong\Hom_R(\Hom_R(M,R),C),
\end{align*}
we see from \cite[Proposition (2.6)(a)]{ABr} that
\begin{align*}
\Ker\lambda_{M\otimes_RC} & \cong\Ext_R^1(\tr M,C)\cong\Ext_R^1(\Omega^tk,C)\\
& \cong\Ext_R^{t+1}(k,C).
\end{align*}
Thus we obtain $\Ext_R^{t+1}(k,C)=0$.
By \cite[Theorem (1.1)]{FFGR}, the $R$-module $C$ must have finite injective dimension.

As we observed in the proof of the implication (1) $\Rightarrow$ (2), assertion (1) implies that $R$ is Cohen-Macaulay.
Thus the last assertion follows.
\qed
\end{mpf}

Now we give applications of our main theorem.
Letting $C=R$ in Theorem \ref{main}, we obtain the following result.
This improves Theorem \ref{foxby} and extends Theorem \ref{ausbr}.

\begin{cor}\label{ques}
The following are equivalent:
\begin{enumerate}[\rm (1)]
\item
$R$ is Gorenstein.
\item
Every $R$-module can be embedded in an $R$-module of finite projective dimension.
\end{enumerate}
\end{cor}

Combining Corollary \ref{ques} with \cite[Theorem 1]{F}, we have the following.

\begin{cor}
If every finitely generated $R$-module can be embedded in a finitely generated $R$-module of finite projective dimension, then every finitely generated $R$-module can be embedded in a finitely generated $R$-module of finite injective dimension.
\end{cor}

\section*{Acknowledgments}
The authors thank Sean Sather-Wagstaff and the referees for their kind comments and valuable suggestions.



\begin{thebibliography}{99}

\bibitem{ABr}
{\sc M. Auslander}; {\sc M. Bridger},
Stable module theory.
Memoirs of the American Mathematical Society, No. 94,
{\it American Mathematical Society, Providence, R.I.}, 1969.

\bibitem{BH}
{\sc W. Bruns}; {\sc J. Herzog},
Cohen-Macaulay rings. revised edition.
Cambridge Studies in Advanced Mathematics, 39,
{\it Cambridge University Press, Cambridge}, 1998.

\bibitem{FFGR}
{\sc R. Fossum}; {\sc H.-B. Foxby}; {\sc P. Griffith}; {\sc I. Reiten},
Minimal injective resolutions with applications to dualizing modules and Gorenstein modules.
{\it Inst. Hautes \'{E}tudes Sci. Publ. Math.} No. 45 (1975), 193--215.

\bibitem{F}
{\sc H.-B. Foxby},
Embedding of modules over Gorenstein rings.
{\it Proc. Amer. Math. Soc.} {\bf 36} (1972), 336--340.

\bibitem{G}
{\sc E. S. Golod},
G-dimension and generalized perfect ideals. (Russian) Algebraic geometry and its applications.
{\it Trudy Mat. Inst. Steklov.} {\bf 165} (1984), 62--66.

\bibitem{S}
{\sc R. Y. Sharp},
Gorenstein modules.
{\it Math. Z.} {\bf 115} (1970), 117--139.

\bibitem{tor}
{\sc R. Takahashi},
A new approximation theory which unifies spherical and Cohen-Macaulay approximations.
{\it J. Pure Appl. Algebra} {\bf 208} (2007), no. 2, 617--634.

\end{thebibliography}
\end{document}